\documentclass{birkjour}
 \newtheorem{thm}{Theorem}[section]
 \newtheorem{cor}[thm]{Corollary}
 \newtheorem{lem}{Lemma}[section]
 
 \theoremstyle{definition}
 \newtheorem{defn}{Definition}[section]
 \theoremstyle{remark}
 
 \newtheorem*{ex}{Example}
 \numberwithin{equation}{section}

\begin{document}

%
%
%
%
%
%
%
%
%

\title[Definite Integral of $\alpha$-Fractal Functions]
 {Definite Integral of $\alpha$-Fractal Functions}

\author[Md Nazimul Islam]{Md Nazimul Islam}

\address{%
Department of Mathematics and Statistics \\
Aliah University \\
IIA/27, New Town, Kolkata-700 160, India.}

\email{n.islam000@gmail.com}

\thanks{}
\author{Imrul Kaish}
\address{Department of Mathematics and Statistics \br
	Aliah University \br
	IIA/27, New Town, Kolkata-700 160, India.}
\email{imrulksh3@gmail.com}


\date{January 1, 2004}

\begin{abstract}
In this article the integration of the $\alpha$-fractal interpolation function $f^{\alpha }$ corresponding to any continuous function $f$ on a compact interval $I$ of $\mathbb{R}$ is estimated although there is no explicit form of $\alpha$-fractal interpolation function till now. Some results related to the definite integral of $f^{\alpha}$ are established. Also the flipped $\alpha $-fractal function $f_{F}^{\alpha _{F}}$ corresponding to the continuous function $f$ is constructed and a result is proved that relates the definite integrals of the fractal functions $f_{F}^{\alpha _{F}}$ and $f^{\alpha }$.	 
\end{abstract}

\maketitle
\section{Introduction}
\label{sec-1} 
There are objects in nature whose geometric structure is
very irregular and complicated. Such as surface of broken stone, the
boundary when a drop of ink falls on a paper, trees, system of
blood vessels, mountain ranges, coastlines, smoke etc. Benoit B. Mandelbrot 
\cite{Mandel} named these objects as fractals.

Barnsley \cite{BarnsFE} developed the fractal theory through introducing the
concept of fractal interpolation function (FIF) and constructed the FIF
using Hutchinson's operator \cite{Hutch} on an iterated function system
(IFS), whose attractor is the graph of a continuous function interpolating a
certain data set. FIFs whose graphs are fractals
have been broadly used in approximation theory, interpolation theory,
financial series, computer graphics, signal processing etc.

Later, Navascu\'{e}s \cite{NavaFP,NavaFT}, introduced the concept of $\alpha 
$-fractal interpolation function $f^{\alpha }$ which is a fractal
perturbation corresponding to a continuous function $f\in \mathcal{C}(I)$
defined on a compact interval $I$ of $\mathbb{R}$. It is discussed in
details the theory and applications of $\alpha $-fractal interpolation
function in the literature \cite{NavaFT,NavaSS,NavaFH}. The function $
f^{\alpha }$ is continuous but in general nowhere differentiable.

Barnsley \emph{et. al.} \cite{BarnsHarri} introduced the calculus of FIF and
showed that integral of a FIF is also a FIF. In \cite{Pan}, the fractional
order integral of FIF is discussed and it is shown that FIF can be
integrated on any closed interval $[a,b]\subset \lbrack 0,\infty )$. In the
same paper, it is also proved that the fractional order integral of FIF is
still a FIF on the interval $[0,b](b>0)$.

In \cite{NavaSebNI}, a method for finding the numerical integration of
affine fractal function is mentioned and an upper bound of error in
computation of the integration of an affine fractal function with the
integration of classical one is placed.

In the present paper, in Section~\ref{sec-3}, we estimate the definite
integral of the $\alpha$-fractal function $f^{\alpha }$ corresponding to a
continuous function $f$ on a compact interval $I$. We obtain some results
related to the definite integral of the function $f^{\alpha }$. In
Section~\ref{sec-4}, we constructed the flipped $\alpha $-fractal function $
f_{F}^{\alpha_{F}}$ corresponding to the continuous function $f$ on the
compact interval $I_{F}$ which is the flipped version of the interval $I$
about the $y$-axis. We proved that the definite integral of $f_{F}^{\alpha _{F}}$ on $I_{F}$ is equal to the
definite integral of $f^{\alpha }$ on $I$.

\section{Definitions and Notations}
\label{sec-2}
\subsection{Fractal Interpolation Functions}
\label{subsec-2.2}
\noindent Let $N\geq 2$ be an integer. Consider a set of interpolation points $
D=\{(x_{i},y_{i})\in I\times\mathbb{R}:i=0,1,...,N\}$, where $\Delta :x_{0}<x_{1}<...<x_{N}$ is a partition of
the closed interval $I=[x_{0},x_{N}]$. Set $I_{i}=[x_{i-1},x_{i}]$ for $
i=1,2,...,N$. Let $L_{i}:I\rightarrow I_{i},$ $i=1,2,...,N,$ be contraction
homeomorphisms with
\begin{equation*}
L_{i}(x_{0})=x_{i-1}, L_{i}(x_{N})=x_{i},
\end{equation*}
\begin{equation*}
|L_{i}(x)-L_{i}(y)|\leq a|x-y|,
\end{equation*}
for all $x,y$ in $I$ and for some $0\leq a<1.$\ \ \
\noindent Let $F_{i}:I\times\mathbb{R}\rightarrow \mathbb{R}$, $i=1,2,...,N$, be given continuous functions satisfying the join-up conditions $\ $
\begin{equation*}
F_{i}(x_{0},y_{0})=y_{i-1},F_{i}(x_{N},y_{N})=y_{i},
\end{equation*}%
\begin{equation*}
|F_{i}(x,y)-F_{i}(x,y^{\prime })|\leq |\alpha _{i}||y-y^{\prime }|,
\end{equation*}
for all $x$ in $I$ and for all $y,y^{\prime }$ in $\mathbb{R}$ and for some $0\leq \left\vert \alpha _{i}\right\vert <1,i=1,2,...,N$.

\noindent Define the mappings $W_{i}:I\times \mathbb{R}\rightarrow I_{i}\times \mathbb{R}
;$ $i=1,2,...,N$ by, for all $(x,y)$ $\in I\times \mathbb{R},$
\begin{equation}
W_{i}(x,y)=(L_{i}(x),F_{i}(x,y)).  \label{ifs}
\end{equation}

\begin{thm}
	\label{Barns}\cite{BarnsFI} The iterated function system $\left\langle
	I\times \mathbb{R};W_{i}(x,y):i=1,..,N\right\rangle $ defined in ($\ref{ifs}$) admits a unique
	attractor $G$, where $G$ is the graph of a continuous function $g:I\rightarrow R$ which obeys $g(x_{i})=y_{i}$ for $i=1,2,...,N$.
\end{thm}

\noindent Let $\mathcal{C}^{\ast }(I)=\{f\in \mathcal{C}
(I):f(x_{0})=y_{0},f(x_{N})=y_{N}\}$ and $\mathcal{C}^{\ast \ast }(I)=\{f\in
\mathcal{C}(I):f(x_{i})=y_{i};$ $i=0,1,2,...,N\}.$

\noindent The Read-Bajraktarvic (RB) operator $T:\mathcal{C}^{\ast }(I)\rightarrow
\mathcal{C}^{\ast \ast }(I)$ defined by (see \cite{MassoFSW})
\begin{equation*}
(Tf)(x)=F_{i}(L_{i}^{-1}(x),f(L_{i}^{-1}(x))); x\in I_{i},i=1,2,...,N,
\end{equation*}
is a contraction with contractivity factor $\left\vert \alpha \right\vert
_{\infty }=\max \{\left\vert \alpha _{i}\right\vert :i=1,..,N\}<1.$ Due to
Banach fixed point theorem, $T$ has a unique fixed point $g$ (say) which
also interpolates the points of $D$. This function $g$ is called a fractal interpolation function or simply a fractal function corresponding to the IFS
($\ref{ifs}$) and this unique function $g$ satisfies the fixed point
equation
\begin{equation}
g(x)=F_{i}(L_{i}^{-1}(x),g(L_{i}^{-1}(x))); x\in I_{i},i=1,2,...,N.
\label{dn1}
\end{equation}
The free parameters $\alpha _{i}$, $i=1,2,...,N$ are the vertical scaling factors
of the transformation $W_{i}$ and the corresponding vector $\alpha =(\alpha _{1},\alpha
_{2},...,\alpha _{N})$ is called the scale vector of the IFS ($\ref{ifs}$).

\subsection{\ $\protect\alpha $-Fractal Interpolation Functions}
\label{subsec-2.3}
\noindent Let $f\in \mathcal{C}(I)$. Consider the IFS defined by the iterated mappings%
\begin{equation}
L_{i}(x)=a_{i}x+e_{i}  \label{dn2}
\end{equation}
and
\begin{equation}
F_{i}(x,y)=\alpha _{i}y+f(L_{i}(x))-\alpha _{i}b(x),  \label{dn3}
\end{equation}
where $a_{i}=\frac{x_{i}-x_{i-1}}{x_{N}-x_{0}}$, $e_{i}=\frac{
	x_{N}x_{i-1}-x_{0}x_{i}}{x_{N}-x_{0}}$, $0\leq \left\vert \alpha
_{i}\right\vert <1,i=1,2,...,N$ and $b\in \mathcal{C}(I)$, known as base
function which satisfy $b(x_{0})=f(x_{0})$ , $b(x_{N})=f(x_{N})$.

\noindent Let $f^{\alpha }(=f_{\Delta ,b}^{\alpha })$ be the continuous function whose
graph is the attractor of the IFS represented by ($\ref{dn2}$) and ($\ref
{dn3}$). This $f^{\alpha }$ is called the $\alpha $\emph{-}fractal interpolation function or simply $\alpha$-fractal function
of $f$ with respect to the base function $b$ and the partition $\Delta $.
From ($\ref{dn1}$), $f^{\alpha }$ satisfies the fixed point equation
\begin{equation}
f^{\alpha }(x)=f(x)+\alpha _{i}(f^{\alpha }-b)(L_{i}^{-1}(x)),  \label{dn4}
\end{equation}
for all $x\in I_{i},$ $i=1,2,...,N.$ From ($\ref{dn4}$), it is easy to
deduce that
\begin{equation*}
\left\Vert f^{\alpha }-f\right\Vert _{\infty }\leq \frac{\left\vert \alpha
	\right\vert _{\infty }}{1-\left\vert \alpha \right\vert _{\infty }}
\left\Vert f-b\right\Vert _{\infty }.
\end{equation*}

\noindent For $\alpha =0$, the fractal function $f^{\alpha }$ agrees with $f$. More
discussion of $\alpha $-fractal function for different choices of $b$ can be
found in \cite{NavaFP,NavaNP,NavaChand}.

\section{Definite Integral of $\protect\alpha$-Fractal Functions}
\label{sec-3}
In this section, the definite integral $\int_{x_{0}}^{x_{N}}f^{\alpha }(x)dx$
of the fractal function $f^{\alpha }$ corresponding to any continuous
function $f$ defined on a compact interval $I=[x_{0},x_{N}]$ is estimated
and some results related to the definite integrals of fractal functions are
established.

\begin{thm}
	\label{thrm-fintegral}Let $f^{\alpha }$ be the fractal function of $f\in  
	\mathcal{C}(I)$ with base $b\in \mathcal{C}(I)$, then  
	\begin{equation}
	\int_{x_{0}}^{x_{N}}f^{\alpha }(x)dx=\frac{1}{(1-\lambda )}
	\int_{x_{0}}^{x_{N}}f(x)dx-\frac{\lambda }{(1-\lambda )}
	\int_{x_{0}}^{x_{N}}b(x)dx,  \label{thrm-fintegral1}
	\end{equation}
	where $\lambda =\sum_{i=1}^{N}a_{i}\alpha _{i}.$
\end{thm}

\begin{proof}
	For $x\in I_{i},i=1,...,N,$ the self-referential equation for $f^{\alpha }$:
	\begin{equation*}
	f^{\alpha }(x)=f(x)+\alpha _{i}(f^{\alpha }-b)(L_{i}^{-1}(x)).
	\end{equation*}
	Then
	\begin{equation*}
	\int_{x_{0}}^{x_{N}}f^{\alpha
	}(x)dx=\int_{x_{0}}^{x_{N}}f(x)dx+\sum_{i=1}^{N}\alpha
	_{i}\int_{x_{i-1}}^{x_{i}}(f^{\alpha }-b)(L_{i}^{-1}(x))dx.
	\end{equation*}
	Letting $L_{i}^{-1}(x)=z,$
	\begin{eqnarray*}
		\int_{x_{0}}^{x_{N}}f^{\alpha }(x)dx &=&\int_{x_{0}}^{x_{N}}f(x)dx+\left[
		\sum_{i=1}^{N}a_{i}\alpha _{i}\right] \int_{x_{0}}^{x_{N}}(f^{\alpha
		}-b)(z)dz \\
		&=&\int_{x_{0}}^{x_{N}}f(x)dx+\lambda \int_{x_{0}}^{x_{N}}(f^{\alpha
		}-b)(x)dx,
	\end{eqnarray*}
	where $\lambda =\sum_{i=1}^{N}a_{i}\alpha _{i}.$	
	\noindent Therefore
	\begin{equation*}
	\int_{x_{0}}^{x_{N}}f^{\alpha}(x)dx=\frac{1}{(1-\lambda )}
	\int_{x_{0}}^{x_{N}}f(x)dx-\frac{\lambda }{(1-\lambda )}
	\int_{x_{0}}^{x_{N}}b(x)dx.
	\end{equation*}
\end{proof}

In the following example, we calculate the definite integral of the fractal
function $f^{\alpha }(x)=(x^3+x)^{\alpha}$, where $\alpha
=(0.2,-0.3,0.5,0.3,0.4)$ with base function $b(x)=2x$, over the interval $
[0,1]$. The function $f(x)=x^3+x$ and its fractal function on $[0,1]$ are
shown in Figure~\ref{fig1}.
\begin{ex}
	Let $\Delta :0<0.2<0.4<0.6<0.8<1$ be a partion of $I=[0,1]$, then $a_{i}=0.2,i=1,2,...,5,$ and $\lambda =\sum_{i=1}^{5}a_{i}\alpha _{i}=0.22.$
	Applying Theorem $\ref{thrm-fintegral}$,
	\begin{eqnarray*}
		\int_{0}^{1}(x^3+x)^{\alpha }dx &=&\frac{1}{(1-\lambda )}
		\int_{0}^{1}(x^3+x)dx-\frac{\lambda }{(1-\lambda )}
		\int_{0}^{1}2xdx \\
		&=&\frac{1}{(1-0.22)}\left[\frac{x^4}{4}+\frac{x^2}{2}\right] _{0}^{1}-\frac{0.22}{(1-0.22)}\left[ x^{2}\right] _{0}^{1} \\
		&=&\frac{53}{78}.
	\end{eqnarray*}
\end{ex}

\begin{figure}[h]
	\centering
	\resizebox{0.75\columnwidth}{!}{
		\includegraphics{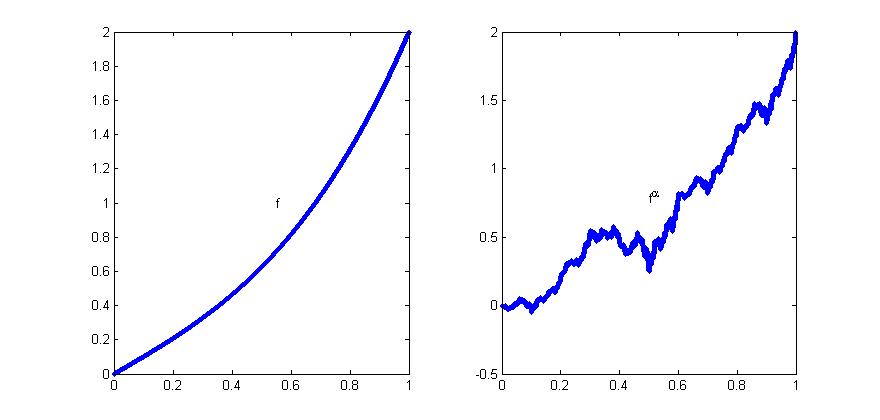} }
	\caption{The function $f(x)=x^3+x$ and its fractal function $f^{\alpha}(x)=(x^3+x)^{\alpha}$}
	\label{fig1}
\end{figure}

\begin{cor}
	\label{cor-sumscale0}Let the scale vector $\alpha =(\alpha _{1},\alpha _{2},...,\alpha _{N})$ be
	such that $\sum_{i=1}^{N}\alpha _{i}=0.$ Then, for any uniform partition $
	\Delta :x_{0}<x_{0}+h<x_{0}+2h<...<x_{0}+Nh=x_{N}$ of $I=[x_{0},x_{N}]$
	and for any base function $b$,
	\begin{equation*}
	\int_{x_{0}}^{x_{N}}f^{\alpha }(x)dx=\int_{x_{0}}^{x_{N}}f(x)dx.
	\end{equation*}
\end{cor}

\begin{proof}
	For uniform partition, $a_{i}=\frac{x_{i}-x_{i-1}}{x_{N}-x_{0}}=\frac{h}{
		x_{N}-x_{0}}=\frac{1}{N},i=1,2,...,N.$ The desired result is obtained by
	setting $\lambda =\sum_{i=1}^{N}a_{i}\alpha _{i}=\frac{1}{N}
	\sum_{i=1}^{N}\alpha _{i}=0$ in Theorem $\ref{thrm-fintegral}$.
\end{proof}

\begin{ex}
	Consider $\Delta :0<0.2<0.4<0.6<0.8<1$ as a partition of $I=[0,1]$. Let $f(x)=\frac{1}{x+1}$ and $b(x)=1-\frac{x}{2}$ be defined on $I=[0,1].$ Let $f^{\alpha }$ be the fractal function of $f$ with base
	function $b$. Suppose the scale vector $\alpha
	=(-0.2,0.4,0.3,-0.6,0.1)$. Here $a_{i}=0.2,i=1,2,...,5,$ and $\lambda
	=\sum_{i=1}^{5}a_{i}\alpha _{i}=0$. Therefore, applying Corollary $\ref{cor-sumscale0}$
	\begin{eqnarray*}
		\int_{0}^{1}f^{\alpha }(x)dx &=&\int_{0}^{1}f(x)dx \\
		&=&\int_{0}^{1}\frac{dx}{x+1} \\
		&=&\log 2.
	\end{eqnarray*}
\end{ex}

\begin{thm}
	\label{thrm-alphabeta}Let $\alpha =(\alpha _{1},\alpha _{2},...,\alpha
	_{N})$ and $\beta =(\beta_{1},\beta_{2},...,\beta_{N})$ be two scale vectors with $\sum_{i=1}^{N}\alpha_{i}=\sum_{i=1}^{N}\beta_{i}$. Suppose that $f^{\alpha }$ and $f^{\beta}$ are respectively the
	$\alpha$-fractal function and $\beta$-fractal function of $f$ with the
	base function $b$. Then, for any uniform partition $\Delta
	:x_{0}<x_{0}+h<x_{0}+2h<...<x_{0}+Nh=x_{N}$ of $I=[x_{0},x_{N}]$,
	\begin{equation*}
	\int_{x_{0}}^{x_{N}}f^{\alpha}(x)dx=\int_{x_{0}}^{x_{N}}f^{\beta}(x)dx.
	\end{equation*}
\end{thm}
\begin{proof}
	For uniform partition, $a_{i}=\frac{x_{i}-x_{i-1}}{x_{N}-x_{0}}=\frac{h}{
		x_{N}-x_{0}}=\frac{1}{N},i=1,2,...,N$ and for the scale vectors $\alpha$ and $\beta$,
	$\lambda_{\alpha}=\sum_{i=1}^{N}a_{i}\alpha_{i}=\frac{1}{N}
	\sum_{i=1}^{N}\alpha_{i}=\frac{1}{N}\sum_{i=1}^{N}\beta_{i}=\sum_{i=1}^{N}a_{i}\beta_{i}=\lambda_{\beta}$.
	
	\noindent Therefore, applying Theorem $\ref{thrm-fintegral}$
	\begin{eqnarray*}
		\int_{x_{0}}^{x_{N}}f^{\alpha}(x)dx &=&\frac{1}{(1-\lambda_{\alpha})}\int_{x_{0}}^{x_{N}}f(x)dx-\frac{\lambda _{\alpha}}{(1-\lambda_{\alpha})}\int_{x_{0}}^{x_{N}}b(x)dx \\
		&=&\frac{1}{(1-\lambda _{\beta })}\int_{x_{0}}^{x_{N}}f(x)dx-\frac{\lambda
			_{\beta}}{(1-\lambda _{\beta})}\int_{x_{0}}^{x_{N}}b(x)dx \\
		&=&\int_{x_{0}}^{x_{N}}f^{\beta}(x)dx.
	\end{eqnarray*}
\end{proof}

\begin{thm}
	If $\left\vert \alpha \right\vert _{\infty }\rightarrow 0,$ then
	\begin{equation*}
	\int_{x_{0}}^{x_{N}}f^{\alpha }(x)dx\rightarrow \int_{x_{0}}^{x_{N}}f(x)dx
	\end{equation*}%
	and if $\alpha =0,$ then
	\begin{equation*}
	\int_{x_{0}}^{x_{N}}f^{\alpha }(x)dx=\int_{x_{0}}^{x_{N}}f(x)dx.
	\end{equation*}
\end{thm}

\begin{proof}
	For first one, use $\lambda =\sum_{i=1}^{N}a_{i}\alpha _{i}\leq
	\sum_{i=1}^{N}a_{i}\left\vert \alpha \right\vert _{\infty }\rightarrow 0$ as
	$\left\vert \alpha \right\vert _{\infty }\rightarrow 0$ in Theorem $\ref
	{thrm-fintegral}$. The second part is straight forward.
\end{proof}

\begin{thm}
	\label{thrm-fLCintegral}Let $f_{b}^{\alpha },g_{\widetilde{b}}^{\alpha }$ be
	the fractal functions of $f,g\in \mathcal{C}(I)$ with the base functions $b,
	\widetilde{b}\in \mathcal{C}(I)$ respectively, then
	\begin{equation*}
	\int_{x_{0}}^{x_{N}}(\gamma f+\delta g)_{\gamma b+\delta \widetilde{b}
	}^{\alpha }(x)dx=\gamma \int_{x_{0}}^{x_{N}}f_{b}^{\alpha }(x)dx+\delta
	\int_{x_{0}}^{x_{N}}g_{\widetilde{b}}^{\alpha }(x)dx,
	\end{equation*}
	for any $\gamma ,\delta \in \mathbb{R}.$
\end{thm}

\begin{proof}
	For $x\in I_{i},i=1,...,N,$ the self-referential equations for $
	f_{b}^{\alpha }$ and $g_{\widetilde{b}}^{\alpha }$ are respectively
	\begin{equation*}
	f_{b}^{\alpha }(x)=f(x)+\alpha _{i}(f_{b}^{\alpha }-b)(L_{i}^{-1}(x))
	\end{equation*}
	and%
	\begin{equation*}
	g_{\widetilde{b}}^{\alpha }(x)=g(x)+\alpha _{i}(g_{\widetilde{b}}^{\alpha }-
	\widetilde{b})(L_{i}^{-1}(x)).
	\end{equation*}
	The equality
	\begin{equation}
	(\gamma f+\delta g)_{\gamma b+\delta \widetilde{b}}^{\alpha }=\gamma
	f_{b}^{\alpha }+\delta g_{\widetilde{b}}^{\alpha }  \label{thrm-fLCintegral1}
	\end{equation}
	is obtained from the uniqueness of the self-referential equation below
	\begin{equation*}
	(\gamma f_{b}^{\alpha }+\delta g_{\widetilde{b}}^{\alpha })(x)=\left( \gamma
	f+\delta g\right) (x)+\alpha _{i}\left[ (\gamma f_{b}^{\alpha }+\delta g_{
		\widetilde{b}}^{\alpha })-(\gamma b+\delta \widetilde{b})\right]
	(L_{i}^{-1}(x)).
	\end{equation*}
	The result is follows from the equation ($\ref{thrm-fLCintegral1}$).
\end{proof}

Setting $\gamma =-1$ and $\delta =0$ in Theorem $\ref{thrm-fLCintegral}$, we get
the following result.

\begin{cor}
	\label{cor-Nfintegral}Let $f_{b}^{\alpha }$ be the fractal function of $f$ with base function $b$, then
	\begin{equation*}
	\int_{x_{0}}^{x_{N}}(-f)_{-b}^{\alpha
	}(x)dx=-\int_{x_{0}}^{x_{N}}f_{b}^{\alpha }(x)dx.
	\end{equation*}
\end{cor}

A supporting example to the Corollary $\ref{cor-Nfintegral}$ is given below.

\begin{ex}
	Let $I=[0,1]$ and $\Delta :0<0.2<0.4<0.6<0.8<1$ be a partition of $I$. Let $
	f(x)=x^{3}$ and $b(x)=x^{2}$ be defined on $I=[0,1].$ Let $f_{b}^{\alpha }$
	and $(-f)_{-b}^{\alpha }$ be the fractal functions of $f$ and $-f$ with base
	functions $b$ and $-b$ respectively. Suppose the scale vector $\alpha
	=(-0.1,0,0.1,0.2,0.3)$. Here $a_{i}=0.2,i=1,2,...,5,$ and $\lambda
	=\sum_{i=1}^{5}a_{i}\alpha _{i}=0.1.$ Therefore, using Theorem $\ref
	{thrm-fintegral}$
	\begin{eqnarray*}
		\int_{0}^{1}(x^{3})_{x^{2}}^{\alpha }(x)dx &=&\int_{0}^{1}f_{b}^{\alpha
		}(x)dx \\
		&=&\frac{1}{(1-\lambda )}\int_{0}^{1}f(x)dx-\frac{\lambda }{(1-\lambda )}%
		\int_{0}^{1}b(x)dx \\
		&=&\frac{1}{(1-0.1)}\int_{0}^{1}x^{3}dx-\frac{0.1}{(1-0.1)}
		\int_{0}^{1}x^{2}dx \\
		&=&\frac{13}{54}.
	\end{eqnarray*}
	and
	\begin{eqnarray*}
		\int_{0}^{1}(-x^{3})_{-x^{2}}^{\alpha }(x)dx
		&=&\int_{0}^{1}(-f)_{-b}^{\alpha }(x)dx \\
		&=&\frac{1}{(1-\lambda )}\int_{0}^{1}(-f(x))dx-\frac{\lambda }{(1-\lambda )}
		\int_{0}^{1}(-b(x))dx \\
		&=&\frac{1}{(1-0.1)}\int_{0}^{1}(-x^{3})dx-\frac{0.1}{(1-0.1)}
		\int_{0}^{1}(-x^{2})dx \\
		&=&-\frac{13}{54}.
	\end{eqnarray*}
\end{ex}

\begin{lem}
	\label{lemma-fcomp}Let $f_{b}^{\alpha }$ be the $\alpha $-fractal function of $f\in
	\mathcal{C}(I)$ with the base function $b\in \mathcal{C}(I).$ If $g\in
	\mathcal{C}(I)$ is linear, then $g\circ f_{b}^{\alpha }$ is also $\alpha $
	-fractal function of $g\circ f$ with the base function $g\circ b$. In other
	words,
	\begin{equation*}
	(g\circ f)_{g\circ b}^{\alpha }=g\circ f_{b}^{\alpha }.
	\end{equation*}
\end{lem}

\begin{proof}
	For $x\in I_{i},i=1,...,N,$ we have
	\begin{equation*}
	f_{b}^{\alpha }(x)=f(x)+\alpha _{i}(f_{b}^{\alpha }-b)(L_{i}^{-1}(x)).
	\end{equation*}
	Then, for $x\in I_{i},i=1,...,N,$
	\begin{eqnarray}
	(g\circ f_{b}^{\alpha })(x) &=&g(f_{b}^{\alpha }(x))  \notag \\
	&=&(g\circ f)(x)+\alpha _{i}(g\circ f_{b}^{\alpha }-g\circ b)(L_{i}^{-1}(x)).
	\label{lemma-fcomp1}
	\end{eqnarray}
	
	\noindent Since $(g\circ f)(x_{0})=(g\circ b)(x_{0})$ and $(g\circ f)(x_{N})=(g\circ
	b)(x_{N})$, the uniqueness of the self-referential equation ($\ref
	{lemma-fcomp1}$) gives the desired result.
\end{proof}

\begin{thm}
	\label{thrm-CwFintegral}Let $f_{b}^{\alpha }$ be the $\alpha $-fractal function of $
	f\in \mathcal{C}(I)$ with the base function $b\in \mathcal{C}(I).$ If $g\in
	\mathcal{C}(I)$ is linear, then
	\begin{equation*}
	\int_{x_{0}}^{x_{N}}(g\circ f_{b}^{\alpha })(x)dx=\frac{1}{(1-\lambda )}
	\int_{x_{0}}^{x_{N}}(g\circ f)(x)dx-\frac{\lambda }{(1-\lambda )}
	\int_{x_{0}}^{x_{N}}(g\circ b)(x)dx,
	\end{equation*}
	where $\lambda =\sum_{i=1}^{N}a_{i}\alpha _{i}.$
\end{thm}

\begin{proof}
	Since $g\circ f_{b}^{\alpha }$ is the $\alpha $-fractal function of $g\circ
	f $ with the base function $g\circ b$, using Theorem $\ref{thrm-fintegral},$
	\begin{eqnarray*}
		\int_{x_{0}}^{x_{N}}(g\circ f_{b}^{\alpha })(x)dx
		&=&\int_{x_{0}}^{x_{N}}(g\circ f)_{g\circ b}^{\alpha }(x)dx \\
		&=&\frac{1}{(1-\lambda )}\int_{x_{0}}^{x_{N}}(g\circ f)(x)dx-\frac{\lambda }{%
			(1-\lambda )}\int_{x_{0}}^{x_{N}}(g\circ b)(x)dx.
	\end{eqnarray*}
\end{proof}

\section{Construction of Flipped $\alpha$-Fractal Function}
\label{sec-4}

\begin{defn}
	Let $I=[x_{0},x_{N}]$ and $I_{F}=[-x_{N},-x_{0}]$. We can define the flipped
	function of $f\in\mathcal{C}(I)$ about $y$-axis as
	\begin{equation*}
	f_{F}(-x)=f(x),\text{for all }x\in I.
	\end{equation*}
	or equivalently
	\begin{equation*}
	f_{F}(x)=f(-x),\text{for all }x\in I_{F}.
	\end{equation*}
\end{defn}
Let $f^{\alpha }(=f_{\Delta ,b}^{\alpha })$ be the $\alpha $-fractal
function associated to $f$ with the base function $b$ and the partition $\Delta $. Recall the fixed point equation for $f^{\alpha }$:
\begin{equation*}
f^{\alpha }(x)=f(x)+\alpha _{i}(f^{\alpha }-b)(L_{i}^{-1}(x)),
\end{equation*}
for all $x\in I_{i},$ $i=1,2,...,N$.

\noindent Let $\Delta _{F}:-x_{N}<-x_{N-1}<...<-x_{0}$ \ is a partition of the closed
interval $I_{F}=[-x_{N},-x_{0}]$. Set $I_{Fi}=[-x_{N+1-i},-x_{N-i}]$ \ for $
i=1,2,...,N$. Let $f_{F}\in \mathcal{C}(I_{F})$ such that $f_{F}(-x)=f(x),$
for all $x\in I$ and let $\alpha _{F}=(\alpha _{N},\alpha _{N-1},...,\alpha
_{1}).$ Define the RB operator $T_{\Delta _{F},b_{F}}:\mathcal{C}^{\ast
}(I_{F})\rightarrow \mathcal{C}^{\ast \ast }(I_{F})$ by
\begin{equation}
(T_{\Delta _{F},b_{F}}g)(x)=f_{F}(x)+\alpha
_{N+1-i}(g-b_{F})((L_{F})_{i}^{-1}(x)),x\in I_{Fi},  \label{rboperator-}
\end{equation}
where $L_{Fi}(x)=a_{Fi}x+e_{Fi},$ $a_{Fi}=a_{N+1-i}$, $
e_{Fi}=-e_{N+1-i},i=1,2,...,N$ and $b_{F}\in \mathcal{C}(I_{F})$ satisfying $
b_{F}(-x)=b(x),$ for all $x\in I$. We can easily show that $T_{\Delta
	_{F},b_{F}}$ is a contraction map.

\noindent Let $f_{F}^{\alpha _{F}}(=f_{F\Delta _{F},b_{F}}^{\alpha _{F}})$ be the $
\alpha _{F}$-fractal function of $f_{F}$ with respect to ($\ref{rboperator-}$
). From ($\ref{rboperator-}$), $f_{F}^{\alpha _{F}}$ satisfies the fixed
point equation
\begin{equation}
f_{F}^{\alpha _{F}}(x)=f_{F}(x)+\alpha _{N+1-i}(f_{F}^{\alpha
	_{F}}-b_{F})((L_{F})_{i}^{-1}(x)),x\in I_{Fi},i=1,2,...,N.  \label{fffe1}
\end{equation}
From ($\ref{fffe1}$), for $i=1,2,...,N$
\begin{equation}
f_{F}^{\alpha _{F}}(-x)=f_{F}(-x)+\alpha _{N+1-i}(f_{F}^{\alpha
	_{F}}-b_{F})((L_{F})_{i}^{-1}(-x)),x\in I_{N+1-i}.  \label{fffe2}
\end{equation}
Since $(L_{F})_{i}^{-1}(-x)=-L_{N+1-i}^{-1}(x),i=1,2,...,N,$ then ($\ref
{fffe2}$) becomes, for $i=1,2,...,N$
\begin{equation}
f_{F}^{\alpha _{F}}(-x)=f_{F}(-x)+\alpha _{N+1-i}(f_{F}^{\alpha
	_{F}}-b_{F})(-L_{N+1-i}^{-1}(x)),x\in I_{N+1-i}.  \notag
\end{equation}
This above equation is the same as
\begin{equation*}
f_{F}^{\alpha _{F}}(-x)=f_{F}(-x)+\alpha _{i}(f_{F}^{\alpha
	_{F}}-b_{F})(-L_{i}^{-1}(x)),x\in I_{i},i=1,2,..,N.
\end{equation*}
Therefore, for $i=1,2,...,N$
\begin{eqnarray*}
	f_{F}^{\alpha _{F}}(-x) &=&f_{F}(-x)+\alpha _{i}f_{F}^{\alpha
		_{F}}(-L_{i}^{-1}(x))-\alpha _{i}b_{F}(-L_{i}^{-1}(x)),x\in I_{i} \\
	&=&f(x)+\alpha _{i}f_{F}^{\alpha _{F}}(-L_{i}^{-1}(x))-\alpha
	_{i}b(L_{i}^{-1}(x)),x\in I_{i}.
\end{eqnarray*}
Let $g$ be the flipped function of $f_{F}^{\alpha _{F}}$, then $f_{F}^{\alpha_{F}}(-x)=g(x),x\in I.$
Now, for all $x\in I_{i},i=1,2,...,N$
\begin{eqnarray*}
	g(x) &=&f_{F}^{\alpha _{F}}(-x) \\
	&=&f(x)+\alpha _{i}f_{F}^{\alpha _{F}}(-L_{i}^{-1}(x))-\alpha
	_{i}b(L_{i}^{-1}(x)) \\
	&=&f(x)+\alpha _{i}g(L_{i}^{-1}(x))-\alpha _{i}b(L_{i}^{-1}(x)) \\
	&=&f(x)+\alpha _{i}(g-b)(L_{i}^{-1}(x)).
\end{eqnarray*}
The uniqueness of the fixed point equation\ of $f^{\alpha }$ implies that
\begin{equation*}
g(x)=f^{\alpha }(x)=f_{F}^{\alpha _{F}}(-x),x\in I_{i},i=1,2,..,N.
\end{equation*}
Thus we see that $f_{F}^{\alpha _{F}}$ is the flipped function of $f^{\alpha }$ about $y$-axis. We say that $f_{F}^{\alpha _{F}}$ is the flipped $\alpha $-fractal
function of $f$.

\begin{ex}
	Figure~\ref{fig2} corresponds to the function $f(x)=\sqrt{x}$ on the interval $
	[0,1] $ and its flipped function $(\sqrt{x})_{F}$ on $[-1,0]$. The fractal function $
	f^{\alpha }$ associated with $f$ corresponding to the partition $\Delta
	:0<0.2<0.4<0.6<0.8<1$, the scaling vector $\alpha =(0.3,0.5,0.2,0.15,0.02)$, and
	the base function $b(x)=x$ and the flipped $\alpha $-fractal function of $f$
	are shown in Figure~\ref{fig3}.
\end{ex}

\begin{figure}[h]
	\centering
	\resizebox{0.74\columnwidth}{!}{
		\includegraphics{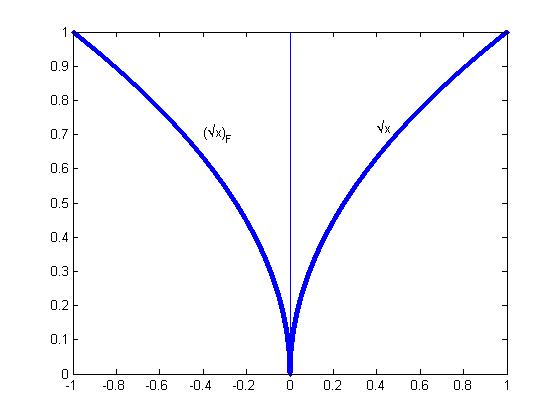} }
	\caption{The function $\protect\sqrt{x}$ and its flipped function $(\protect\sqrt{x})_{F}$ }
	\label{fig2}
\end{figure}
\begin{figure}[h]
	\centering
	\resizebox{1.71\columnwidth}{!}{
		\includegraphics{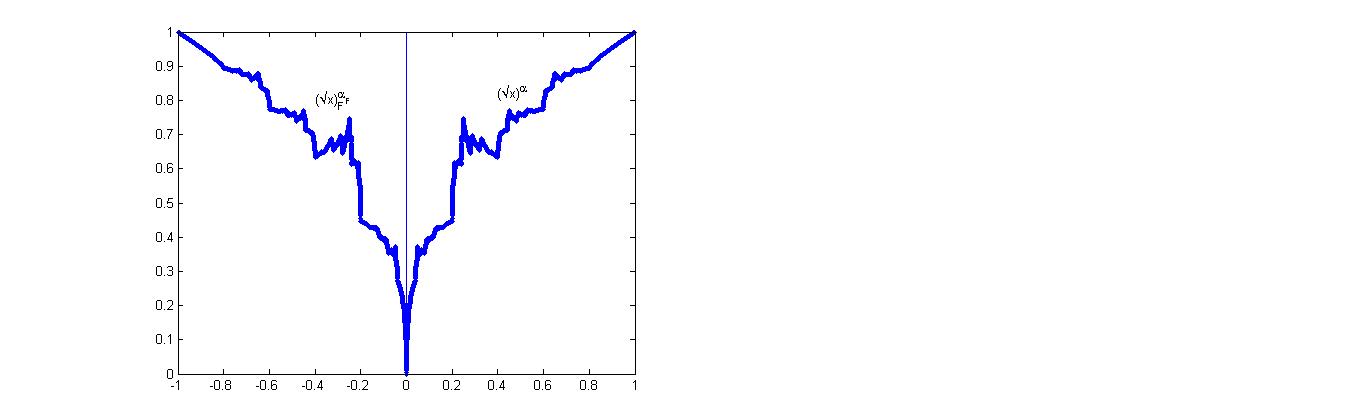} }
	\caption{Fractal function $(\protect\sqrt{x})^{\protect\alpha }$ and its flipped function $(\protect\sqrt{x})_{F}^{{\protect\alpha }_{F}}$}
	\label{fig3}
\end{figure}

\begin{thm}
	\label{thrm-ffintegral}Suppose that $I=[x_{0},x_{N}]$ be any compact
	interval. Let $f^{\alpha }$ be the fractal function of $f\in \mathcal{C}(I)$
	with base $b\in \mathcal{C}(I)$, then
	\begin{equation*}
	\int_{-x_{N}}^{-x_{0}}f_{F}^{\alpha _{F}}(x)dx=\int_{x_{0}}^{x_{N}}f^{\alpha
	}(x)dx.
	\end{equation*}
\end{thm}

\begin{proof}
	For $x\in I_{Fi},i=1,...,N,$ the self-referential equation for $
	f_{F}^{\alpha _{F}}$ is
	\begin{equation*}
	f_{F}^{\alpha _{F}}(x)=f_{F}(x)+\alpha _{N+1-i}(f_{F}^{\alpha
		_{F}}-b_{F})((L_{F})_{i}^{-1}(x)).
	\end{equation*}
	Then%
	\begin{eqnarray*}
		\int_{-x_{N}}^{-x_{0}}f_{F}^{\alpha _{F}}(x)dx
		&=&\int_{-x_{N}}^{-x_{0}}f_{F}(x)dx \\
		&&+\sum_{i=1}^{N}\alpha _{N+1-i}\int_{-x_{N+1-i}}^{-x_{N-i}}(f_{F}^{\alpha
			_{F}}-b_{F})((L_{F})_{i}^{-1}(x))dx.
	\end{eqnarray*}
	Letting $(L_{F})_{i}^{-1}(x)=z,$
	\begin{eqnarray*}
		\int_{-x_{N}}^{-x_{0}}f_{F}^{\alpha _{F}}(x)dx
		&=&\int_{-x_{N}}^{-x_{0}}f_{F}(x)dx \\
		&&+\left[ \sum_{i=1}^{N}a_{N+1-i}\alpha _{N+1-i}\right]
		\int_{-x_{N}}^{-x_{0}}(f_{F}^{\alpha _{F}}-b_{F})(z)dz \\
		&=&\int_{-x_{N}}^{-x_{0}}f_{F}(x)dx+\lambda
		_{F}\int_{-x_{N}}^{-x_{0}}(f_{F}^{\alpha _{F}}-b_{F})(x)dx,
	\end{eqnarray*}
	where $\lambda _{F}=\sum_{i=1}^{N}a_{N+1-i}\alpha
	_{N+1-i}=\sum_{i=1}^{N}a_{i}\alpha _{i}=\lambda .$
	
	\noindent Therefore
	\begin{eqnarray*}
		\int_{-x_{N}}^{-x_{0}}f_{F}^{\alpha _{F}}(x)dx &=&\frac{1}{(1-\lambda )}
		\int_{-x_{N}}^{-x_{0}}f_{F}(x)dx-\frac{\lambda }{(1-\lambda )}
		\int_{-x_{N}}^{-x_{0}}b_{F}(x)dx \\
		&=&\frac{1}{(1-\lambda )}\int_{-x_{N}}^{-x_{0}}f(-x)dx-\frac{\lambda }{
			(1-\lambda )}\int_{-x_{N}}^{-x_{0}}b(-x)dx \\
		&=&\frac{1}{(1-\lambda )}\int_{x_{0}}^{x_{N}}f(x)dx-\frac{\lambda }{
			(1-\lambda )}\int_{x_{0}}^{x_{N}}b(x)dx \\
		&=&\int_{x_{0}}^{x_{N}}f^{\alpha}(x)dx.
	\end{eqnarray*}
\end{proof}

The following example is an application of the Theorem $\ref{thrm-ffintegral}
$.

\begin{ex}
	Consider the integral $\int_{-1}^{0}(x^{2})_{-x}^{\alpha }(x)dx,$ where $\alpha =(0.2,-0.1,0,0.3,0.4).$
	To apply the Theorem $\ref{thrm-ffintegral}$, first we\ need to find out $
	\int_{0}^{1}(x^{2})_{x}^{\alpha }(x)dx.$ For this, let $\Delta
	:0<0.2<0.4<0.6<0.8<1$ be a partition of $I=[0,1]$, then $
	a_{i}=0.2,i=1,2,...,5,$ and $\lambda =\sum_{i=1}^{5}a_{i}\alpha _{i}=0.16.$
	Using Theorem $\ref{thrm-fintegral}$ and applying the Theorem $\ref
	{thrm-ffintegral}$,
	\begin{eqnarray*}
		\int_{-1}^{0}(x^{2})_{-x}^{\alpha }(x)dx &=&\int_{0}^{1}(x^{2})_{x}^{\alpha
		}(x)dx \\
		&=&\frac{1}{(1-\lambda )}\int_{0}^{1}x^{2}dx-\frac{\lambda }{(1-\lambda )}
		\int_{0}^{1}xdx \\
		&=&\frac{1}{3(1-0.16)}\left[ x^{3}\right] _{0}^{1}-\frac{0.16}{2(1-0.16)}\left[
		x^{2}\right] _{0}^{1} \\
		&=&\frac{19}{63}.
	\end{eqnarray*}
\end{ex}

\end{document}